\documentclass[11pt]{article}
\usepackage{latexsym}
\usepackage{amssymb,amsmath,amsfonts}
\usepackage[mathscr]{eucal}
\usepackage{amscd}
\usepackage{amsfonts,cite}
\usepackage{epsfig, amsmath, amsthm, amssymb}
\usepackage{latexsym}
\usepackage{graphics,epstopdf}
\newtheorem{theorem}{Theorem}[section]

\newtheorem{corollary}{Corollary}[section]

\newtheorem{remark}{Remark}[section]

\numberwithin{equation}{section}

\begin{document}

\begin{center}
{\Large{\bf Certain hybrid polynomials associated with Sheffer sequences}}\\~~

{\bf Nabiullah Khan, Talha Usman and  Mohd Aman${}^{*}$, \\
{\footnote {\textit{key words and phrases}. Legendre-Gould Hopper based Sheffer polynomials; Monomiality principle; Operational techniques.}
 }}\\

\end{center}

\parindent=8mm
\noindent
\begin{abstract}
Inspired by the framework of operational methods and based on the generating functions of Legendre-Gould Hopper polynomials and Sheffer sequences, we discuss certain new mixed type polynomials and their important properties. We show that the use of operational nature allows the relevant polynomials to be unified and general in nature. It is illustrated how the polynomials, we develop, provide an easy derivation of a wide class of new and known polynomials, and their respective properties.
\end{abstract}

\noindent

\noindent
{\bf{\em Keywords:}}~~Legendre-Gould Hopper polynomials, Sheffer sequences; Monomiality principle; Operational techniques. \\

\section{Introduction and preliminaries}
Let the two types of series $f(t)$ and $g(t)$ be delta series and invertible series  respectively. Let us consider a polynomial sequence $ s_{n} (x)$ satisfying the conditions
\begin{equation}\label{1.6}
\big<g(t)~f(t)^{k}|~s_{n}(x)\big>=n!~\delta_{n,k}~,\quad \forall~ n, k\ge 0.
\end{equation}
This sequence is called the Sheffer sequence for $(g(t),f(t))$ which is usually denoted as $s_{n}(x)\sim(g(t)\,f(t))$. If $s_{n}(x)\sim(1,f(t))$, then $s_{n}(x)$ is called the associated sequence for $f(t)$. If $s_{n}(x)\sim(g(t),t)$, then $s_{n}(x)$ is called the Appell sequence (see \cite [ p. 17]{2SRom}).
The polynomial sequence $\{s_{n} (x)\}_{n=0}^{\infty }$ ($ s_{n} (x)$ being a polynomial of degree $n$) is also called  Sheffer A-type zero \cite [ p.222 (Theorem 72)]{2Rain}, (which we shall hereafter call Sheffer-type), if $s_{n} (x)$ whose exponential generating function is of the form
\begin{equation}\label{1.1}
A(t)\exp\left(xH(t)\right)=\sum _{n=0}^{\infty}s_{n} (x)\frac{t^{n}}{n!},
\end{equation}
where $A(t)$ and $H(t)$ have (at least the formal) expansions:
\begin{equation}\label{1.2}
A(t)=\sum _{n=0}^{\infty }A_{n} \frac{t^{n} }{n!},\quad A_{0} \ne 0
\end{equation}
and
\begin{equation}\label{1.3}
H(t)=\sum _{n=1}^{\infty }H_{n} \frac{t^{n} }{n!},\quad H_{1} \ne 0,
\end{equation}



According to Roman \cite [p.18 (Theorem 2.3.4)]{2SRom},  the polynomial sequence $s_{n}(x)$  
has the generating function
\begin{equation}\label{1.7}
\frac {1}{g\big(f^{-1}(t)\big)} \exp \big(xf^{-1}(t)\big)=\sum _{n=0}^{\infty }s_{n} (x)\frac{t^{n} }{n!},
\end{equation}
for all $x$ in $\mathbb C$, where $f^{-1}(t)$ is the compositional inverse of $f(t)$.\\

In view of equations \eqref{1.1} and \eqref{1.7}, we have
\begin{equation}\label{1.8}
A(t)=\frac {1}{g\big(f^{-1}(t)\big)}
\end{equation}
and
\begin{equation}\label{1.9}
H(t)=f^{-1}(t).
\end{equation}
The idea of Quasi-Monomial (QM) treatment of special polynomials emerged within the context of poweroid, suggested by J. F. Steffenson \cite{2JFStef}. The QM principle has been proved to be a powerful tool for the investigation of the properties of a wide class of polynomials like Sheffer polynomials, see for example \cite{2GGDat}. The series of papers developed during the last years, for example the work of Dattoli (see \cite{2Dat}) deepend their roots in QM theory and reformulated the so called study of special polynomials.\\
 According to such a point of view, we recall the following definitions:\\
A family of polynomials $p_{n}(x),~({n \in \mathbb{N},~\forall x\in\mathbb{R}})$ is said to be a QM, if a couple of operators, $\hat{M}$ and $\hat{P}$, hereafter called respectively, the multiplicative and  derivative operators, do exist and satisfy the rules:
\begin{equation}\label{mcap}
\hat{M}\, \{p_{n} (x)\}=p_{n+1} (x)
\end{equation}
and
\begin{equation}\label{pcap}
\hat{P}\, \{p_{n} (x)\}=n ~p_{n-1} (x).
\end{equation}
For $n \in \mathbb{N}, x\in\mathbb{R}$, the combinations of the two identities gives 
\begin{equation}\label{mpcap}
\hat{M}\,\hat{P} \{p_{n} (x)\}=n~p_{n} (x)
\end{equation}
and 
\begin{equation}
\hat{P}\,\hat{M} \{p_{n} (x)\}=(n+1)~p_{n} (x),
\end{equation}
which allows the commutation relation of the operators to be
\begin{equation}\label{p,m}
[\hat{P},\hat{M}]=\hat{P}\hat{M}-\hat{M}\hat{P}=\hat{1}
\end{equation}
and thus can be viewed as the generators of a Weyl group structure.

Assuming hereafter, the vacuum is such that  $p_{0} (x)=1$, then, from \eqref{mcap}, the family of polynomials $p_{n}(x)$ can be can be generated as:
\begin{equation}\label{1.14}
\hat{M}^{n}\{1\}=p_{n}(x),\quad \forall n\in\mathbb{N},~\forall x\in\mathbb{R}.
\end{equation} 
The generating function of $p_{n}(x)$ straightforwardly follows from the identity \eqref{1.14}, namely
\begin{equation}\label{1.15}
\sum _{n=0}^{\infty }p_{n}(x)\frac{t^{n}}{n!}=\sum _{n=0}^{\infty }\frac{t^{n}\hat{M}^{n}}{n!}\{1\}=\exp{(t\hat{M})}\{1\},\quad|t| < \infty.
\end{equation}
The previous definitions summed up briefly, gives enough content of QM theory to exploit their use on special polynomials.

\vspace{.25cm}
We recall the Gould Hopper polynomials [\cite{2HWG}, p. 58, (6.2)] (also called sometimes as higher-order Hermite or Kamp\'{e} de F\'{e}riet polynomials) $H_{n}^{(s)(x,y)}$, which are defined as
\begin{equation}\label{GHP}
H_{n}^{(s)}(x,y)=n!\sum_{k=0}^{[\frac{n}{s}]}\frac{y^{k}x^{n-sk}}{k!\,(n-sk)!},
\end{equation}
and their generating function reads
\begin{equation}\label{GHPgen}
\exp(xt+yt^{s})=\sum_{n=0}^{\infty}\,H_{n}^{(s)}(x,y)\frac{t^{n}}{n!}.
\end{equation}
They appear as the solution of the generalized heat equation (see \cite{2SGDat})
\begin{equation}\label{heat}
\frac{\partial}{\partial y}\,f(x,y)=\frac{\partial^{s}}{\partial{x}^{s}}\,f(x,y)\quad f(x,0)=x^{n},
\end{equation}
and under the operational formalism, they are defined as
\begin{equation}\label{1.11}
H_{n}^{(s)}(x,y)=\exp\left(y\frac{\partial^{s}}{\partial{x}^{s}}\right)\{x^{n}\}.
\end{equation}
These polynomials are quasi-monomial under the action of the operators (see \cite{2SGDat})
\begin{equation}\label{GHPoper}
\aligned
&\hat{M}_{H}:=x+sy\frac{\partial^{s-1}}{\partial x^{s-1}},\\
&\hat{P}_{H}:=\frac{\partial}{\partial x}.
\endaligned
\end{equation}
Now we recall the the Legendre polynomials $S_{n}(x,y)$ and $\frac{R_{n}(x,y)}{n!}$ which are defined by the generating functions (see \cite{2GPED}):
\begin{equation}\label{leg1}
\exp(yt)\,C_{0}(-xt^{2})=\sum_{n=0}^{\infty}S_{n}(x,y)
\frac{t^{n}}{n!}
\end{equation}
and 
\begin{equation}\label{leg2}
C_{0}(xt)\,C_{0}(-yt)=\sum_{n=0}^{\infty}\,\frac{R_{n}
(x,y)}{n!}\frac{t^{n}}{n!},
\end{equation}
respectively, where $C_0(x)$ denotes the Bessel-Tricomi function of order $0$. The $nth$ order Bessel-Tricomi function $C_{n}(x)$ defined by the following series \cite [p.150]{2Dat}:
\begin{equation}\label{tricomi}
C_n(x)=x^{-\frac{n}{2}}J_n(2 \sqrt {x})~=~\sum_{k=0}^\infty \frac{(-1)^k\ x^k}{k!\ (n+k)!},\quad n=0,1,2,\ldots\ ,
\end{equation}

with $J_n(x)$ being the ordinary cylindrical Bessel function of first kind \cite{2And}. The operational definition of $0^{th}$-order Bessel-Tricomi function $C_{0}(x)$ is given by:
\begin{equation}\label{tricomi-exp}
C_{0}(\alpha x)=\exp \left(-\alpha D_{x}^{-1}\right)\{1\},
\end{equation}
where $D_{x}^{-1}$ denotes the inverse derivative operator and
\begin{equation}
D_{x}^{-n}\{1\}=\frac{x^{n}}{n!}.
\end{equation}
Lately, by taking as base, the Legendre polynomials $S_{n}(x,y)$ and $\frac{R_{n}(x,y)}{n!}$ in the generating function \eqref{GHPgen} of Gould Hopper polynomials (GHP), the author in \cite{ghazala} introduced and studied the Legendre-Gould Hopper polynomials (LeGHP) ${}_S{H}_n^{(s)}(x,y,z)$ and ${}_R{H}_n^{(s)}(x,y,z)$, which are defined, respectively, by the following generating function:
\begin{equation}\label{leGHP1}
C_0(-xt^2)\exp(yt+zt^{r})=\sum_{n=0}^{\infty}{}_SH_n^{(r)}(x,y,z)\frac{t^n}{n!}
\end{equation}
and
\begin{equation}\label{leGHP2}
C_0(xt)\,C_0(-yt)\,\exp(zt^{r})=\sum_{n=0}^{\infty}\frac{{}_RH_n^{(r)}(x,y,z)}{n!}\frac{t^n}{n!}.
\end{equation}
Systematically, they were found to be quasi-monomial under the action of operators
\begin{equation}\label{leghpcap1}
\aligned
&\hat{M}_{SH}:=y+2D_x^{-1}\frac{\partial}{\partial y}+rz\frac{\partial^{r-1}}{\partial y^{r-1}},\\
&\hat{P}_{SH}=\frac{\partial}{\partial y}
\endaligned
\end{equation}
and
\begin{equation}\label{leghpcap2}
\aligned
&\hat{M}_{RH}:=-D_x^{-1}+D_y^{-1}+rz\frac{\partial^{r-1}}{\partial y^{r-1}},\\
&\hat{P}_{RH}=-\frac{\partial}{\partial x}x\frac{\partial}{\partial x},
\endaligned
\end{equation}
respectively.

\vspace{.25cm}
The LeGHP ${}_SH_{n}^{(s)}(x,y,z)$ and $\frac{{}_RH_{n}^{(s)}(x,y,z)}{n!}$ contain several special polynomials as their particular cases. Their generality can be viewed from the below table:\\

\noindent
\textbf{Table 1. Particular cases of the LeGHP ${}_SH_{n}^{(s)}(x,y,z)$ and $\frac{{}_RH_{n}^{(s)}(x,y,z)}{n!}$ }\\
\\
{\tiny{
\begin{tabular}{lllll}
\hline
\bf{S.} & \bf{Values of } & \bf{Relation Between the }& \bf{Name of }& \bf{Series}  \\
\bf{No.}& \bf{the Indices} & \bf{LeGHP~${}_SH_{n}^{(r)}(x,y,z)$,$\frac{{}_RH_{n}^{(r)}(x,y,z)}{n!}$ }&\bf{the Known}& \bf{Definitions} \\
& \bf{and Variables}& \bf{and its Special Case}&  \bf{Polynomials}&\\
\hline
 I.&  &  &  &  \\
& $x=0$ & ${}_SH_{n}^{(r)}(0,y,z)={}_SH_{n}^{(r)}(y,z)$ & Gould Hopper \cite{2GGGA} &${}_SH_{n}^{(r)}(x,y)=n!\sum_{k=0}^{\frac{n}{r}}\frac{y^{k}x^{n-rk}}{k!(n-rk)!}$\\
&&& &\\
\hline
II. &  & &2-Variable &  \\
&~$z=0$  &${}_SH_n^{(r)}(x,y,0)={}_{2L}L_{n}(x,y)$ & Legendre - &${}_{2L}L_{n}(x,y)=n!\sum_{s=0}^{\frac{n}{2}}\frac{x^{s}y^{n-2s}}{(s!)^{2}(n-2s)!}$ \\
&&&type \cite{2GAMDat}&\\
\hline
 III.&i. $r=m;~x=0,$  & & 2-variable -&  \\
& ~$y\rightarrow-D_{x}^{-1}$,$z\rightarrow y$  & ${}_SH_{n}^{(m)}(0,D_{x}^{-1},y)={}_{[m]}L_n(x,y)$& generalized Lagurre&${}_{[m]}L_n(x,y)=n!\sum_{k=0}^{\frac{n}{m}}\frac{y^{k}x^{n-mk}}{k![(n-mk)!]^{2}}$ \\
&ii.$r=m; y=0,z\rightarrow y$& ${}_RH_{n}^{(m)}(x,0,y)={}_{[m]}L_n(x,y)$ & type \cite{2GGADat}&\\
\hline
IV.& $r=m-1;~x=0,$  & & generalized&  \\
& ~$y\rightarrow x$,$z\rightarrow y$  & ${}_SH_{n}^{(m-1)}(0,x,y)=U_n^{(m)}(x,y)$& Chebyshev \cite{2HWG}&$U_n^{(m)}(x,y)=\sum_{k=0}^{\frac{n}{m}}\frac{(n-k)!y^{k}x^{n-mk}}{k!(n-mk)!}$ \\
 &\\
\hline
V.&i. $r=1;~x=0,z\rightarrow-D_{x}^{-1}$  & & 2-variable -&  \\
& ~$y\rightarrow-D_{x}^{-1}$,$z\rightarrow y$  & ${}_SH_{n}^{(1)}(0,y,-D_{x}^{-1})=L_n(x,y)$&  Lagurre  \cite{2SGDat}&$L_n(x,y)=n!\sum_{s=0}^{n}\frac{(-x)^{s}y^{n-s}}{(s!)^{2}(n-s)!}$ \\
&ii.$r=1; y=0,z\rightarrow y$& ${}_RH_{n}^{(1)}(x,0,y)=L_n(x,y)$&\\
\hline
VI. &  & &2-Variable &  \\
&~$z=0$  &${}_RH_n^{(r)}(x,y,0)=R_{n}(x,y)$ & Legendre \cite{2GDat} &$R_{n}(x,y)=(n!)^{2}\sum_{s=0}^{n}\frac{y^{s}(-x)^{n-s}}{(s!)^{2}[(n-2s)!]^{2}}$ \\
&&&\\
\hline
VII. & $x=0,y\rightarrow x, z\rightarrow yD_{y}y$, & ${}_SH_{n}^{(r)}(0,x,yD_{y}y)=e_n^{(r)}(x,y)$ &truncated of order r &$e_n^{(r)}(x,y)=n!\sum_{k=0}^{\frac{n}{r}}\frac{x^{n-rk}y^{K}}{(n-rk)!}$  \\
&&2-variable \cite{2DAT}&\\
\hline
VIII. & $r=2,x=0$, & ${}_SH_{n}^{(2)}(0,y,z)=H_n(y,z)$ &2-variable &$H_n(x,y)=n!\sum_{k=0}^{\frac{n}{2}}\frac{x^{n-2k}y^{K}}{k!(n-2k)!}$  \\
&&&Hermite Kamp\'{e} de F\'{e}riet \cite{2GADat}&\\
\hline
IX.&i. $r=2$; $x=0$,  & ${}_SH_{n}^{(2)}(0,-D_x^{-1},y)=G_{n}(x,y)$ & Hermite&  \\
  & $y\rightarrow D_{x}^{-1},z\rightarrow y$&${}_RH_{n}^{(2)}(0,x,y)=G_{n}(x,y)$&type \cite{2GPED} &  $G_{n}(x,y)=n!\sum_{k=0}^{\frac{n}{2}} \frac{y^{k}x^{n-2k}}{k![(n-2k)!]^{2}}$\\
 &$ii. r=2; x=0$,&\\&$y\rightarrow x,z\rightarrow
 y$&&&\\
\hline
X.&i. $x\rightarrow\frac{(x^{2}-1)}{4}$,  & ${}_SH_{n}^{(r)}(\frac{(x^{2}-1)}{4},x,0)=P_{n}(x,y)$   \\
& $y\rightarrow x, z=0$&${}_RH_{n}^{(1)}(\frac{(1-x)}{2},\frac{(x+1)}{2},0)=P_{n}(x,y)$ & Legendre \cite{2GDat} & $P_{n}(x,y)=n!\sum_{k=0}^{\frac{n}{2}} \frac{(x^{2}-1)^{k}x^{n-2k}}{2^{2k}(k!)^{2}(n-2k)!}$\\
&$ii. r=1; x\rightarrow\frac{(1-x)}{2}$,&\\&$y\rightarrow \frac{(1+x)}{2},z=
0$&&&\\

\hline
XI.& $r=3$; $x\rightarrow zD_{z}z$,  & ${}_SH_{n}^{(3)}(zD_{z}z,x,y)=H_{n}^{(3,2)}(x,y,z)$ & Bell-type \cite{2GDattoli}& $ H_{n}^{(3,2)}(x,y,z)=n!\sum_{k=0}^{\frac{n}{3}}\frac{y^{k}H_{n-3k}^{(2)}(x,y)}{k!r!(n-3k-2r)!}$ \\
  & $y\rightarrow x$, $z\rightarrow y$ & &&$=n!\sum_{k=0}^{\frac{n}{3}}\sum_{s=0}^{\frac{n-3k}{2}}\frac{y^{k}z^{s}x^{n-3k-2s}}{k!s!(n-3k-2s)!}$ \\
\hline

\end{tabular}}}\\
\\

The special cases mentioned in Table 1 will be exploited later to derive the relevant results.\\

In this paper, the families of Legendre-Gould Hopper based Sheffer polynomials are introduced by using the concepts and the methods associated with monomiality principle. In Section 2, we introduce the  Legendre-Gould Hopper based Sheffer polynomials (LeGHSP) ${}_{{}_{s}LeG^{(r)}}s_{n}(x,y,z)$ and frame these polynomials within the context of monomiality principle formalism. In Section 3, we consider some examples of these polynomials. In Section 4, the operational and integral representations  for the Laguerre-Gould Hopper based Sheffer polynomials are established. In Section 5, results are obtained for the members of Legendre-Gould Hopper based Sheffer and Legendre-Gould Hopper based associated Sheffer polynomial families by considering some members of the Sheffer and associated Sheffer families respectively. \\
\\

\section{Legendre-Gould Hopper based Sheffer polynomials}

To introduce the Legendre-Gould Hopper based Sheffer polynomials ( LeGHSP) denoted by ${}_{{}_{s}LeG^{(r)}}s_{n}(x,y,z)$ and ${}_{{}_{r}LeG^{(r)}}s_{n}(x,y,z)$ we prove the following results:\\

\begin{theorem}\label{t2.1}
The generating function of the Legendre-Gould Hopper based Sheffer polynomials ${}_{{}_{s}LeG^{(r)}}s_{n}(x,y,z)$ reads 
\begin{equation}\label{2.1a}
\frac {1}{g\big(f^{-1}(t)\big)}C_0\left(-x\big(f^{-1}(t)\big)^2\right)\exp\left(yf^{-1}(t)+z\big(f^{-1}(t)\big)^r\right)=\sum_{n=0}^{\infty}{}_{{}_{s}LeG^{(r)}}s_{n}(x,y,z)\frac{t^n}{n!},
\end{equation}
or
\begin{equation}\label{2.1b}
A(t)C_0\left(-x\big(H(t)\big)^2\right)\exp\left(yH(t)+z\big(H(t)\big)^r\right)=\sum_{n=0}^{\infty}{}_{{}_{s}LeG^{(r)}}s_{n}(x,y,z)\frac{t^n}{n!}.
\end{equation}
\end{theorem}
\begin{proof}
We follow the operational technique of replacing $x$ in equation \eqref{1.7} by the multiplicative operator ${}_s\hat {M}_{LeG}$ of the LeGHP ${}_S{H}_n^{(s)}(x,y,z)$, we have indeed
\begin{equation}\label{2.3}
\frac {1}{g\big(f^{-1}(t)\big)} \exp\big({}_s\hat {M}_{LeG}f^{-1}(t)\big)=\sum _{n=0}^{\infty }s_{n}\big({}_s\hat {M}_{LeG}\big) \frac{t^{n} }{n!},
\end{equation}
which by using the expression in \eqref{leghpcap1} and the Crofton-type identity \cite [p. 12]{2GPDat}
\begin{equation}\label{crofton}
f\left(y+m\lambda\frac{d^{m-1}}{dy^{m-1}}\right)\big\{1\big\}=\exp\left(\lambda\frac{d^m}{dy^m}\right)\Big\{f(y)\Big\},
\end{equation}
yields
\begin{equation}
\frac {1}{g\big(f^{-1}(t)\big)}\exp{\Big(z\frac{\partial^{r}}{\partial y^{r}}\Big)} \exp{\left(\left(y+2D_x^{-1}\frac{\partial}{\partial y}\right)f^{-1}(t)\right)}=\sum _{n=0}^{\infty }s_{n}\Big(y+2D_x^{-1}\frac{\partial}{\partial y}+rz\frac{\partial^{r-1}}{\partial y^{r-1}}\Big) \frac{t^{n} }{n!}.
\end{equation}
Further using the identity \eqref{crofton} gives
\begin{equation}\label{2.6}
\frac {1}{g\big(f^{-1}(t)\big)}\exp{\Big(z\frac{\partial^{r}}{\partial y^{r}}\Big)}\exp{\Big(D_x^{-1}\frac{\partial^{2}}{\partial y^{2}}\Big)} \exp({yf^{-1}(t)})=\sum _{n=0}^{\infty }s_{n}\Big(y+2D_x^{-1}\frac{\partial}{\partial y}+rz\frac{\partial^{r-1}}{\partial y^{r-1}}\Big) \frac{t^{n} }{n!}.
\end{equation}
Now, using the exponential expansion and \eqref{tricomi-exp} in \eqref{2.6}, gives
\begin{equation}
\frac {1}{g\big(f^{-1}(t)\big)}C_0\left(-x\big(f^{-1}(t)\big)^2\right)\exp{\Big(z\frac{\partial^{r}}{\partial y^{r}}\Big)} \exp({yf^{-1}(t)})=\sum _{n=0}^{\infty }s_{n}\Big(y+mD_x^{-1}\frac{\partial}{\partial y}+rz\frac{\partial^{r-1}}{\partial y^{r-1}}\Big) \frac{t^{n} }{n!},
\end{equation}
which, on further simplification and by denoting the resultant LeGHSP by ${}_{{}_{s}LeG^{(r)}}s_{n}(x,y,z)$, namely
\begin{equation}\label{2.8}
{}_{{}_{s}LeG^{(r)}}s_{n}(x,y,z)=s_{n}({}_s\hat{M}_{LeG})=s_{n}\Big(y+2D_x^{-1}\frac{\partial}{\partial y}+rz\frac{\partial^{r-1}}{\partial y^{r-1}}\Big),
\end{equation}
the assertion \eqref{2.1a} directly follows. Also, the equivalent generating function in \eqref{2.1b} follows easily in view of equations \eqref{1.8} and \eqref{1.9} .
\end{proof}
\begin{theorem}
	The generating function of Legendre-Gould Hopper based Sheffer polynomials $\frac{{}_{{}_{r}LeG^{(r)}}s_{n}(x,y,z)}{n!}$ reads 
\begin{equation}\label{2.9}
\frac {1}{g\big(f^{-1}(t)\big)}C_0\left(xf^{-1}(t)\right)C_{0}(-yf^{-1}(t))\exp\left(z\big(f^{-1}(t)\big)^r\right)=\sum_{n=0}^{\infty}{}_{{}_{r}LeG^{(r)}}s_{n}(x,y,z)\frac{t^n}{(n!)^{2}},
\end{equation}
	or
	\begin{equation}\label{2.10}
	A(t)C_0\left(xf^{-1}(t)\right)C_{0}(-yf^{-1}(t))\exp\left(z\big(H(t)\big)^r\right)=\sum_{n=0}^{\infty}{}_{{}_{r}LeG^{(r)}}s_{n}(x,y,z)\frac{t^n}{(n!)^{2}}.
	\end{equation}
\end{theorem}
\begin{proof}
	Following the same procedure as of Theorem \ref{t2.1}, that is, using the multiplicative operator ${}_r\hat{M}_{LeG}$ of the LeGHP $\frac{{{}_{r}LeG^{(r)}}s_{n}(x,y,z)}{(n!)}$ given in \eqref{leghpcap2}, we rewrite generating function \eqref{1.7} as
	\begin{equation}
	\frac {1}{g\big(f^{-1}(t)\big)} \exp\big({}_r\hat {M}_{LeG}f^{-1}(t)\big)=\sum _{n=0}^{\infty }s_{n}\big({}_r\hat {M}_{LeG}\big) \frac{t^{n} }{n!}.
	\end{equation}
	Using the expression of ${}_r\hat{M}_{LeG}$ given in equation \eqref{leghpcap2} and then decoupling the exponential operator in the l.h.s. of the resultant equation by using the Crofton-type identity \cite [p. 12]{2GPDat}
	\begin{equation}
	f\left(y+m\lambda\frac{d^{m-1}}{dy^{m-1}}\right)\big\{1\big\}=\exp\left(\lambda\frac{d^m}{dy^m}\right)\Big\{f(y)\Big\},
	\end{equation}
	After some caculations, we find
	\begin{equation}
	\frac {1}{g\big(f^{-1}(t)\big)}C_0\left(xf^{-1}(t)\right)C_{0}(-yf^{-1}(t))\exp{\Big(z\frac{\partial^{r}}{\partial y^{r}}\Big)} \exp({f^{-1}(t)})=\sum _{n=0}^{\infty }s_{n}\Big(-D_x^{-1}+D_y^{-1}+rz\frac{\partial^{r-1}}{\partial y^{r-1}}\Big) \frac{t^{n} }{n!},
	\end{equation}
	from which the assertion \eqref{leghpcap2} follows. Also, in view of equations \eqref{1.8} and \eqref{1.9}, generating function \eqref{leghpcap2} can be expressed equivalently as equation \eqref{2.10}.
\end{proof}
In order to show that the LeGHSP ${}_{{}_{s}LeG^{(r)}}s_{n}(x,y,z)$ and $\frac{{{}_{r}LeG^{(r)}}s_{n}(x,y,z)}{n!}$ satisfy the monomiality principle, we prove the following results:
\begin{theorem}
The Legendre-Gould Hopper based Sheffer polynomials ${}_{{}_{s}LeG^{(r)}}s_{n}(x,y,z)$ are quasi-monomial under the action of the following multiplicative and derivative operators:
\begin{equation}\label{2.14}
{}_s\hat {M}_{LeGs}=\left(y+2D_x^{-1}\frac{\partial}{\partial y}+rz\frac{\partial^{r-1}}{\partial y^{r-1}}-\frac {g^{\prime} \left(\partial_{y}\right)} {g\left(\partial_{y}\right)}\right)\frac {1} {f^{\prime}\left(\partial_{y}\right)},
\end{equation}
or
\begin{equation}\label{2.15}
{}_s\hat {M}_{LeGs}=\left(y+2D_x^{-1}\frac{\partial}{\partial y}+rz\frac{\partial^{r-1}}{\partial y^{r-1}}+\frac {A^{\prime} \left(H^{-1}\left(\partial_{y}\right)\right)}{A\left(H^{-1}\left(\partial_{y}\right)\right)}\right)H'\left(H^{-1}\left(\partial_{y}\right)\right)
\end{equation}
and
\begin{equation}\label{2.16}
{}_s\hat {P}_{LHs}=f\left(\partial_{y}\right),
\end{equation}
or
\begin{equation}\label{2.17}
 {}_s\hat {P}_{LHs}=H^{-1}\left(\partial_{y}\right),
\end{equation}
respectively, where $\partial_{y}:= \frac{\partial}{\partial y}$.
\end{theorem}
\begin{proof}
Consider the following identity:
\begin{equation}\label{dy}
\partial_{y}~\left\{\exp \left(yf^{-1}(t)+z\big(f^{-1}(t)\big)^r\right)\right\}=f^{-1}(t)~\exp \left(yf^{-1}(t)+z\big(f^{-1}(t)\big)^r\right).
\end{equation}

Since $f^{-1}$ denotes the compositional inverse of the function $f$ and $f(t)$ has an expansion (1.3a) in powers of $t$, therefore we have
\begin{equation}\label{dy2}
f\left(\partial_{y}\right)~\left\{\exp\left(yf^{-1}(t)+z\big(f^{-1}(t)\big)^r\right)\right\}=t~\exp\left(yf^{-1}(t)+z\big(f^{-1}(t)\big)^r\right).
\end{equation}

Differentiating equation \eqref{2.3} partially with respect to $t$ and in view of relation \eqref{2.8}, we find
\begin{equation}
\left(\left({}_s\hat {M}_{LeG}-\frac{g'\left(f^{-1}(t)\right)}{g\left(f^{-1}(t)\right)}\right) \frac{1}{f'(f^{-1}(t))}\right)\frac {1}{g\big(f^{-1}(t)\big)} \exp{\big({}_s\hat {M}_{LeG} f^{-1}(t)\big)}=\sum _{n=0}^{\infty }{}_{{}_{s}LeG^{(r)}}s_{n+1}(x,y,z)\frac{t^{n} }{n!},
\end{equation}
which on using monomiality principle equation \eqref{1.15} with $t=f^{-1}(t)$ gives
\begin{equation}
\aligned
\left(\left(\hat {M}_{LeG}-\frac{g'\left(f^{-1}(t)\right)}{g\left(f^{-1}(t)\right)}\right) \frac{1}{f'(f^{-1}(t))}\right)&\frac {1}{g\big(f^{-1}(t)\big)}C_0\left(-x\big(f^{-1}(t)\big)^2\right)\exp\left(yf^{-1}(t)+z\big(f^{-1}(t)\big)^r\right)\\
=&\sum _{n=0}^{\infty }{}_{{}_{s}LeG^{(r)}}s_{n+1}(x,y,z)\frac{t^{n} }{n!}.
\endaligned
\end{equation}

Since $g(t)$ is an invertible series and $f(t)$ is a delta series of $t$ therefore $\frac{g'\left(f^{-1}(t)\right)}{g\left(f^{-1}(t)\right)}$ and $\frac{1}{f'(f^{-1}(t))}$ possess power series expansions of $f^{-1}(t)$. Thus, in view of relation \eqref{dy}, the above equation becomes
\begin{equation}
\aligned
\left(\left({}_s\hat {M}_{LeG}-\frac{g'\left(\partial_{y}\right)}{g\left(\partial_{y}\right)}\right) \frac{1}{f'(\partial_{y})}\right)&\left\{\frac {1}{g\big(f^{-1}(t)\big)}C_0\left(-x\big(f^{-1}(t)\big)^2\right)\exp\left(yf^{-1}(t)+z\big(f^{-1}(t)\big)^r\right)\right\}\\
=&\sum _{n=0}^{\infty }{}_{{}_{s}LeG^{(r)}}s_{n+1}(x,y,z)\frac{t^{n} }{n!},
\endaligned
\end{equation}
which on using generating function \eqref{2.1a} becomes
\begin{equation}
\left(\left({}_s\hat {M}_{LeG}-\frac{g'\left(\partial_{y}\right)}{g\left(\partial_{y}\right)}\right) \frac{1}{f'(\partial_{y})}\right)\left \{\sum _{n=0}^{\infty }{}_{{}_{s}LeG^{(r)}}s_{n}(x,y,z)\frac{t^{n} }{n!}\right \}=\sum _{n=0}^{\infty }{}_{{}_{s}LeG^{(r)}}s_{n+1}(x,y,z)\frac{t^{n} }{n!},
\end{equation}
or, in an equal manner, it becomes
\begin{equation}
\sum _{n=0}^{\infty}\left(\left({}_s\hat {M}_{LeG}-\frac{g'\left(\partial_{y}\right)}{g\left(\partial_{y}\right)}\right) \frac{1}{f'(\partial_{y})}\right)\left\{{}_{{}_{s}LeG^{(r)}}s_{n}(x,y,z)\right\}\frac{t^{n} }{n!}=\sum _{n=0}^{\infty }{}_{{}_{s}LeG^{(r)}}s_{n+1}(x,y,z)\frac{t^{n} }{n!}.
\end{equation}
Now, equating the coefficients of like powers of $t$ in the above equation, we find
\begin{equation}
\left(\left({}_s\hat {M}_{LeG}-\frac{g'\left(\partial_{y}\right)}{g\left(\partial_{y}\right)}\right) \frac{1}{f'(\partial_{y})}\right)\{{}_{{}_{s}LeG^{(r)}}s_{n}(x,y,z)\}={}_{{}_{s}LeG^{(r)}}s_{n+1}(x,y,z),
\end{equation}
which, in view of equation \eqref{mcap} shows that the multiplicative operator for ${}_{{}_{s}LeG^{(r)}}s_{n}(x,y,z)$ is given as:
$${}_s\hat{M}_{LeGs}=\left({}_s\hat {M}_{LeG}-\frac{g'\left(\partial_{y}\right)}{g\left(\partial_{y}\right)}\right) \frac{1}{f'(\partial_{y})}.$$

Finally, using equation \eqref{leghpcap1} in the r.h.s of above equation, we get assertion \eqref{2.9}. \\

Again, in view of identity \eqref{dy2}, we have
\begin{equation}
\aligned
f \left(\partial_{y}\right)& \left\{\frac {1}{g\big(f^{-1}(t)\big)}C_0\left(-x\big(f^{-1}(t)\big)^2\right)\exp\left(yf^{-1}(t)+z\big(f^{-1}(t)\big)^r\right)\right\}\\
=&t \frac {1}{g\big(f^{-1}(t)\big)}C_0\left(-x\big(f^{-1}(t)\big)^2\right)\exp\left(yf^{-1}(t)+z\big(f^{-1}(t)\big)^r\right),
\endaligned
\end{equation}
which on using generating function \eqref{2.1a} becomes
\begin{equation}
f \left(\partial_{y}\right) \left \{\sum _{n=0}^{\infty }{}_{{}_{s}LeG^{(r)}}s_{n}(x,y,z)\frac{t^{n} }{n!}\right \}=\sum _{n=1}^{\infty }{}_{{}_{s}LeG^{(r)}}s_{n-1}(x,y,z)\frac{t^{n} }{(n-1)!},
\end{equation}
or
\begin{equation}
\sum _{n=0}^{\infty }f \left(\partial_{y}\right)  \{{}_{{}_{s}LeG^{(r)}}s_{n}(x,y,z)\} \frac{t^{n} }{n!}=\sum _{n=1}^{\infty }{{}_{{}_{s}LeG^{(r)}}s_{n-1}(x,y,z)}\frac{t^{n} }{(n-1)!}.
\end{equation}
Equating the coefficients of like powers of $t$ in the above equation, we get
\begin{equation}
f \left(\partial_{y}\right)  \{{}_{{}_{s}LeG^{(r)}}s_{n}(x,y,z)\}=n~{}_{{}_{s}LeG^{(r)}}s_{n}(x,y,z),\quad(n \ge 1),
\end{equation}
which in view of equation \eqref{pcap} yields assertion \eqref{2.16}.  Also, in view of relations \eqref{1.8} and \eqref{1.9}, assertions \eqref{2.14} and \eqref{2.16} can be expressed equivalently as equations \eqref{2.15} and \eqref{2.17}, respectively.
\end{proof}
\begin{theorem}
	The Legendre-Gould Hopper based Sheffer polynomials $\frac{{{}_{r}LeG^{(r)}}s_{n}(x,y,z)}{n!}$ are quasi-monomial under the action of the following multiplicative and derivative operators:
	\begin{equation}\label{2.30}
	{}_r\hat {M}_{LeGs}=\left(-D_x^{-1}+D_y^{-1}+rz\frac{\partial^{r-1}}{\partial y^{r-1}}-\frac {g^{\prime} \left(\partial_{y}\right)} {g\left(\partial_{y}\right)}\right)\frac {1} {f^{\prime}\left(\partial_{y}\right)},
	\end{equation}
	and
	\begin{equation}\label{2.31}
	 {}_r\hat {P}_{LHs}=f\left(\frac{-\partial}{\partial x}x\frac{\partial}{\partial_{y}}\right),
	\end{equation}
	respectively, where $\partial_{y}:= \frac{\partial}{\partial y}$.
\end{theorem}
\begin{remark}
In view of equation \eqref{1.14} and using equations \eqref{2.9} and \eqref{2.10}, we deduce the following consequence of Theorem 2.2.
\end{remark}
\begin{corollary}
The Legendre-Gould Hopper based Sheffer polynomials ${}_{{}_{s}LeG^{(r)}}s_{n-1}(x,y,z)$ have the following explicit representations:
\begin{equation*}
{}_{{}_{s}LeG^{(r)}}s_{n}(x,y,z)={}_s\hat {M}_{LeGs}^n\{1\}
\end{equation*}
that is
\begin{equation}
{}_{{}_{s}LeG^{(r)}}s_{n}(x,y,z)=\left(\left(y+2D_x^{-1}\frac{\partial}{\partial y}+rz\frac{\partial^{r-1}}{\partial y^{r-1}}-\frac {g^{\prime} \left(\partial_{y}\right)} {g\left(\partial_{y}\right)}\right)\frac {1} {f^{\prime}\left(\partial_{y}\right)}\right)^{n}\{1\},\end{equation}
or,
\begin{equation}
{}_{{}_{s}LeG^{(r)}}s_{n}(x,y,z)=\left(\left(y+2D_x^{-1}\frac{\partial}{\partial y}+rz\frac{\partial^{r-1}}{\partial y^{r-1}}\right)H'\left(H^{-1}\left(\partial_{y}\right)\right)+\frac {A^{\prime} \left(H^{-1}\left(\partial_{y}\right)\right)}{A\left(H^{-1}\left(\partial_{y}\right)\right)}\right)^{n}\{1\}.
\end{equation}
\end{corollary}
\begin{theorem}
The Legendre-Gould Hopper based Sheffer polynomials ${}_{{}_{s}LeG^{(r)}}s_{n}(x,y,z)$ and $\frac{{{}_{r}LeG^{(r)}}s_{n}(x,y,z)}{n!}$ satisfy the following respective differential equation:
\begin{equation}\label{2.34}
\left(\left(y+2D_x^{-1}\frac{\partial}{\partial y}+rz\frac{\partial^{r-1}}{\partial y^{r-1}}-\frac {g^{\prime} \left(\partial_{y}\right)} {g\left(\partial_{y}\right)}\right)\frac {f \left(\partial_{y}\right)} {f^{\prime}\left(\partial_{y}\right)}-n\right){}_{{}_{s}LeG^{(r)}}s_{n}(x,y,z)=0
\end{equation}
and
\begin{equation}\label{2.35}
\aligned
\left(\left(-D_x^{-1}+D_y^{-1}+rz\frac{\partial^{r-1}}{\partial y^{r-1}}-\frac {g^{\prime} \left(\partial_{y}\right)} {g\left(\partial_{y}\right)}\right)\frac {f\left(\frac{-\partial}{\partial x}x\frac{\partial}{\partial_{y}}\right)} {f^{\prime}\left(\partial_{y}\right)}-n\right){}_{{}_{r}LeG^{(r)}}s_{n}(x,y,z)=0.
\endaligned
\end{equation}
\end{theorem}
\noindent
{\bf{Proof.}} Using equations \eqref{2.14} and \eqref{2.16} in equation \eqref{mpcap}, we get assertion \eqref{2.34} and similarly using equations \eqref{2.30} and \eqref{2.31} in equation \eqref{mpcap}, we get assertion \eqref{2.35}.
\begin{remark}
Since the Sheffer sequence $s_{n}(x)$ for $g(t)=1$ becomes the associated Sheffer sequence ${\mathfrak s}_{n}(x)$ for $f(t)$. (For our convenience, we denote the associated Sheffer sequence by ${\mathfrak s}_{n}(x)$). Therefore, for $g(t)=1$, we deduce the following consequences of Theorems 2.1-2.5:
\end{remark}
\begin{corollary}
The Legendre-Gould Hopper based associated Sheffer polynomials (LeGHASP) ${}_{{}_{s}LeG^{(r)}}\mathfrak{s}_{n}(x,y,z)$ and $\frac{{{}_{r}LeG^{(r)}}s_{n}(x,y,z)}{n!}$ are defined by the generating function respectively as
\begin{equation}
C_0\left(-x\big(f^{-1}(t)\big)^{2}\right)\exp\left(yf^{-1}(t)+z\big(f^{-1}(t)\big)^r\right)=\sum_{n=0}^{\infty}{}_{{}_{s}LeG^{(r)}}\mathfrak{s}_{n}(x,y,z)\frac{t^n}{n!},
\end{equation}
or
\begin{equation}
C_0\left(-x\big(H(t)\big)^2\right)\exp\left(yH(t)+z\big(H(t)\big)^r\right)=\sum_{n=0}^{\infty}{}_{{}_{s}LeG^{(r)}}\mathfrak{s}_{n}(x,y,z)\frac{t^n}{n!}
\end{equation}
and
\begin{equation}
C_0\left(xf^{-1}(t)\right)C_{0}(-yf^{-1}(t))\exp\left(z\big(f^{-1}(t)\big)^r\right)=\sum_{n=0}^{\infty}{}_{{}_{r}LeG^{(r)}}\mathfrak{r}_{n}(x,y,z)\frac{t^n}{(n!)^{2}},
\end{equation}
or
\begin{equation}
C_0\left(xH(t)\right)C_{0}(-yH(t))\exp\left(z\big(H(t)\big)^r\right)=\sum_{n=0}^{\infty}{}_{{}_{r}LeG^{(r)}}\mathfrak{r}_{n}(x,y,z)\frac{t^n}{(n!)^{2}}.
\end{equation}
\end{corollary}

\begin{corollary}
The Legendre-Gould Hopper based associated Sheffer polynomials ${}_{{}_{s}LeG^{(r)}}\mathfrak{s}_{n}(x,y,z)$ and $\frac{{{}_{r}LeG^{(r)}}s_{n}(x,y,z)}{n!}$ are quasi-monomial under the action of the following multiplicative and derivative operators:
\begin{equation}
{}_s\hat {M}_{LeGs}=\left(y+2D_x^{-1}\frac{\partial}{\partial y}+rz\frac{\partial^{r-1}}{\partial y^{r-1}}\right)\frac {1} {f^{\prime}\left(\partial_{y}\right)},
\end{equation}
\begin{equation}
\hat {P}_{LH{\mathfrak s}}=f\left(\partial_{y}\right)
\end{equation}
and
\begin{equation}
{}_r\hat {M}_{LeGs}=\left(-D_x^{-1}+D_y^{-1}+rz\frac{\partial^{r-1}}{\partial y^{r-1}}\right)\frac {1} {f^{\prime}\left(\partial_{y}\right)},
\end{equation}
\begin{equation}
{}_r\hat {P}_{LHs}=f\left(\frac{-\partial}{\partial x}x\frac{\partial}{\partial_{y}}\right),
\end{equation}

respectively.
\end{corollary}
\begin{corollary}
The Legendre-Gould Hopper based associated Sheffer polynomials ${}_{{}_{s}LeG^{(r)}}\mathfrak{s}_{n}(x,y,z)$ and $\frac{{{}_{r}LeG^{(r)}}s_{n}(x,y,z)}{n!}$ satisfy the differential equations
\begin{equation}
\left(\left(y+2D_x^{-1}\frac{\partial}{\partial y}+rz\frac{\partial^{r-1}}{\partial y^{r-1}}\right)\frac {f \left(\partial_{y}\right)} {f^{\prime}\left(\partial_{y}\right)}-n\right){}_{{}_{s}LeG^{(r)}}\mathfrak{s}_{n}(x,y,z)=0
\end{equation}

and
\begin{equation}
\left(\left(-D_x^{-1}+D_y^{-1}+rz\frac{\partial^{r-1}}{\partial y^{r-1}}\right)\frac {f\left(\frac{-\partial}{\partial x}x\frac{\partial}{\partial_{y}}\right)} {f^{\prime}\left(\partial_{y}\right)}-n\right){}_{{}_{r}LeG^{(r)}}\mathfrak{r}_{n}(x,y,z)=0
\end{equation}
respectively.
\end{corollary}
\begin{remark}
Since, for $f(t)=t$, the Sheffer polynomials $s_{n}(x)$ reduce to the Appell polynomials $A_{n}(x)$ \cite{2App}. Therefore, by taking $f(t)=t$ in Theorems 2.1-2.5, we can obtain the corresponding results for the Legendre-Gould Hopper based Appell polynomials (LeGHAP). 
\end{remark}

In the next section, we consider certain new and known families of special polynomials related to the Sheffer sequences and obtain the results for these mixed type special polynomials. \\

\vspace{.50cm}
\section{Examples}

In Table 1, we have mentioned special cases of the LGHP ${}_{L}H_{n}^{(m,r)}(x,y,z)$. In order to obtain the results for the corresponding new or known special polynomials related to the Sheffer sequences, we consider the following examples:

\noindent
{\bf Example 1.} Since, for $x=0$, the LeGHP ${}_SH_n^{(r)}(x,y,z)$ reduce to the Gould Hopper polynomials (GHP) $H_{n}^{(r)}(x,y,z)$ (Table 1(I)). Therefore, for the same choice of $x$, the  LeGHSP ${}_{{}_{s}LeG^{(r)}}s_{n}(x,y,z)$ reduce to the Gould Hopper based Sheffer polynomials (GHSP) ${}_SH_{n}^{(r)}(x,y,z)$. Thus, by using these substitutions in Theorems 2.1, 2.3 and 2.5, we get the corresponding results for the GHSP ${}_SH_{n}^{(r)}(x,y,z)$ (see \cite{2SubMum})\\

\noindent
{\bf Example 2.} For  $z=0$, the LeGHP ${}_SH_n^{(r)}(x,y,z)$ reduce to the 2-variable Legendre-type polynomials (2VLebP) ${}_{2L}L_{n}(x,y)$ (Table 1(II)). Therefore, for the same choice of $z$, the  LeGHSP ${}_{{}_{s}LeG^{(r)}}s_{n}(x,y,z)$ reduce to the 2-variable Legendre-type based Sheffer polynomials (2VLebSP) $_{{}_2Leb}s_{n}(x,y)$. Thus, by using these substitutions in Theorems 2.1, 2.3 and 2.5, we get the following results for the 2VLebSP $_{{}_2Leb}s_{n}(x,y)$:

\vspace{.25cm}
\noindent
\noindent \textbf{Table 1. Results for the 2VLebSP $_{{}_2Leb}s_{n}(x,y)$}\\
\\
{\tiny{
		\begin{tabular}{lll}
			\hline
			\bf{S.} && \\
			\bf{No.} &~~~~~~~~~{\bf Results}&~~~~~~~~~~~~~~~~~~~~~~~{\bf Mathematical Expressions}\\
			
			\hline
			\bf{1.} & \bf{Generating } & $\frac {1}{g\big(f^{-1}(t)\big)}C_0\left(x\big(f^{-1}(t)\big)^{2}\right)\exp\left(yf^{-1}(t)\right)$ \\
			\bf{}& \bf{function} &  $~~~~~~~~~~~~~=A(t)C_0\left(-x\big(H(t)\big)^{2}\right)\exp\left(yH(t)\right)=\sum_{n=0}^{\infty}\,_{{}_2Leb}s_{n}(x,y)\frac{t^n}{n!}$\\
			
			\hline
			\bf{2.} & \bf{Multiplicative and } & $\hat {M}=\left(y+2D_x^{-1}\frac{\partial}{\partial y}-\frac {g^{\prime} \left(\partial_{y}\right)} {g\left(\partial_{y}\right)}\right)\frac {1} {f^{\prime}\left(\partial_{y}\right)}$\\
			\bf{}& \bf{derivative operators} &  $~~~~~~~~~~=\left(y+2D_x^{-1}\frac{\partial}{\partial y}\right)H'\left(H^{-1}\left(\partial_{y}\right)\right)+\frac {A^{\prime} \left(H^{-1}\left(\partial_{y}\right)\right)}{A\left(H^{-1}\left(\partial_{y}\right)\right)}$,~~~~$\hat {P}=f\left(\partial_{y}\right)=H^{-1}\left(\partial_{y}\right)$\\
			
			\hline
			\bf{3.} & \bf{Differential } & $\left(\left(y+2D_x^{-1}\frac{\partial}{\partial y}-\frac {g^{\prime} \left(\partial_{y}\right)} {g\left(\partial_{y}\right)}\right)\frac {f \left(\partial_{y}\right)} {f^{\prime}\left(\partial_{y}\right)}-n\right)\,_{{}_2Leb}s_{n}(x,y)=0$ \\
			\bf{}& \bf{equation} & \\

			\hline
			
			\bf{4.} & \bf{Explicit} & $_{{}_2Leb}s_{n}(x,y)=\left(\left(y+2D_x^{-1}\frac{\partial}{\partial y}-\frac {g^{\prime} \left(\partial_{y}\right)} {g\left(\partial_{y}\right)}\right)\frac {1} {f^{\prime}\left(\partial_{y}\right)}\right)^{n}\{1\}$\\
			\bf{}& \bf{representation} &  ~~~~~~~~~~~~~~~~~~~~~~~~~~~~~~\\
			
			\hline
			
\end{tabular}}}\\
\\

\noindent
{\bf Example 3.} Since, for $r=m$, $x=0$, $y\rightarrow -D_{x}^{-1}$, $z\rightarrow y$ and  $r=m$; $y=0$,$z\rightarrow y$, the LeGHP ${}_SH_n^{(r)}(x,y,z)$ and $\frac{{}_RH_n^{(r)}(x,y,z)}{n!}$ respectively reduce to the 2-variable generalized Laguerre type polynomials (2gLTP) ${}_{[m]}L_n(x,y)$ (Table 1(III)). Therefore, for the same choice of $r$, $x$, $y$ and $z$, the  LeGHSP ${}_{{}_{s}LeG^{(r)}}s_{n}(x,y,z)$ reduce to the 2-variable generalized Laguerre type based Sheffer polynomials (2gLTSP) ${_{{}_{[m]}L}s_{n}(x,y)}$. Thus, by using these substitutions in Theorems 2.1, 2.3 and 2.5, we get the following results:

\noindent \textbf{Table 2. Results for the 2gLTSP ${_{{}_{[m]}L}s_{n}(x,y)}$}\\
\\
{\tiny{
		\begin{tabular}{lll}
			\hline
			\bf{S.} && \\
			\bf{No.} &~~~~~~~~~{\bf Results}&~~~~~~~~~~~~~~~~~~~~~~~{\bf Mathematical Expressions}\\
			
			\hline
			\bf{1.} & \bf{Generating } & $\frac {1}{g\big(f^{-1}(t)\big)}C_0\left(xf^{-1}(t)\right)\exp\left(y\big(f^{-1}(t)\big)^m\right)$ \\
			\bf{}& \bf{function} &  $~~~~~~~~~~~~~~~~~~~~~~~~~~~~=A(t)C_0\left(xH(t)\right)\exp\left(y\big(H(t)\big)^m\right)=\sum_{n=0}^{\infty}{_{{}_{[m]}L}s_{n}(x,y)}\frac{t^n}{n!}$\\
			
			\hline
			\bf{2.} & \bf{Multiplicative and } & $\hat {M}=\left(-D_x^{-1}+(-1)^{m}my\frac{\partial^{m-1}}{\partial x^{m-1}}x^{m-1} \frac{\partial^{m-1}}{\partial x^{m-1}}-\frac {g^{\prime} \left(-\partial_{x}x\partial_{x}\right)} {g\left(-\partial_{x}x\partial_{x}\right)}\right)\frac {1} {f^{\prime}\left(-\partial_{x}x\partial_{x}\right)}$\\
			\bf{}& &  $~~~=\left(-D_x^{-1}+(-1)^{m}my\frac{\partial^{m-1}}{\partial x^{m-1}}x^{m-1} \frac{\partial^{m-1}}{\partial x^{m-1}}\right)H'\left(H^{-1}\left(-\partial_{x}x\partial_{x}\right)\right)+\frac {A^{\prime} \left(H^{-1}\left(-\partial_{x}x\partial_{x}\right)\right)}{A\left(H^{-1}\left(-\partial_{x}x\partial_{x}\right)\right)}$,\\
			&\bf{derivative operators} &~~~$\hat {P}=f\left(-\partial_{x}x\partial_{x}\right)=H^{-1}\left(-\partial_{x}x\partial_{x}\right)$\\
			
			\hline
			\bf{3.} & \bf{Differential } & $\left(\left(-D_x^{-1}+(-1)^{m}my\frac{\partial^{m-1}}{\partial x^{m-1}}x^{m-1} \frac{\partial^{m-1}}{\partial x^{m-1}}-\frac {g^{\prime} \left(-\partial_{x}x\partial_{x}\right)} {g\left(-\partial_{x}x\partial_{x}\right)}\right)\frac {f \left(-\partial_{x}x\partial_{x}\right)} {f^{\prime}\left(-\partial_{x}x\partial_{x}\right)}-n\right){_{{}_{[m]}L}s_{n}(x,y)}=0,$ or\\
			\bf{}& \bf{equation} &  $\left(\left(\left(y+mD_x^{-1}\frac{\partial^{m-1}}{\partial y^{m-1}}\right) H'\left(H^{-1}\left(-\partial_{x}x\partial_{x}\right)\right) +\frac {A^{\prime} \left(H^{-1}\left(-\partial_{x}x\partial_{x}\right)\right)}{A\left(H^{-1}\left(-\partial_{x}x\partial_{x}\right)\right)}\right) H^{-1}\left(-\partial_{x}x\partial_{x}\right)-n\right)$\\
			&&~~~~~~~~~~~~~~~~~~~~~~~~~~~~~~~~~~~~~~~~~~~~~~~~~~~~~$ {_{{}_{[m]}L}s_{n}(x,y)}=0$\\
			
			\hline
			\bf{4.} & \bf{Explicit} & $_{{}_{[m]}L}s_{n}(x,y)=\left(\left(-D_x^{-1}+(-1)^{m}my\frac{\partial^{m-1}}{\partial x^{m-1}}x^{m-1} \frac{\partial^{m-1}}{\partial x^{m-1}}-\frac {g^{\prime} \left(-\partial_{x}x\partial_{x}\right)} {g\left(-\partial_{x}x\partial_{x}\right)}\right)\frac {1} {f^{\prime}\left(-\partial_{x}x\partial_{x}\right)}\right)^{n}\{1\}$\\
			\bf{}& \bf{representation} &  $=\left(\left(-D_x^{-1}+(-1)^{m}my\frac{\partial^{m-1}}{\partial x^{m-1}}x^{m-1} \frac{\partial^{m-1}}{\partial x^{m-1}}\right)H'\left(H^{-1}\left(-\partial_{x}x\partial_{x}\right)\right)+\frac {A^{\prime} \left(H^{-1}\left(-\partial_{x}x\partial_{x}\right)\right)}{A\left(H^{-1}\left(-\partial_{x}x\partial_{x}\right)\right)}\right)^{n}\{1\}$\\
			
			\hline
			
\end{tabular}}}\\
\\

\noindent
{\bf Example 4.}  Since, for $r=m-1$; $x=0$, $y\rightarrow x$, $z\rightarrow y$ the LeGHP ${}_SH_n^{(r)}(x,y,z)$ reduce to the generalized Chebyshev polynomials (GCP) $U_{n}^{(m)}(x,y)$ (Table 1(IV)). Therefore, for the same choice of $r$, $x$, $y$ and $z$, the  LeGHSP ${}_{{}_{s}LeG^{(r)}}s_{n}(x,y,z)$ reduce to the generalized Chebyshev based Sheffer polynomials (GCSP) $_{{}_{gC}}s_{n}(x,y)$ \cite{2SubMum}. Thus, by using these substitutions in Theorems 2.1, 2.3 and 2.5, we obtain the corresponding results  for the GCSP $_{{}_{gC}}s_{n}(x,y)$ \cite{2SubMum}.\\\\

\noindent
{\bf Example 5.}  
  Since, for $r=1$;~$x=0$, $z\rightarrow-D_{x}^{-1}$ and for $r=1$; $y=0$, $z\rightarrow y$ the LeGHP ${}_SH_n^{(r)}(x,y,z)$ and $\frac{{}_RH_n^{(r)}(x,y,z)}{n!}$ respectively reduce to the 2-variable Laguerre polynomials (2VLP) $L_n(x,y)$ (Table 1(V)). Therefore, for the same choice of variables and parameters, the  LeGHSP reduce to the 2-variable Laguerre based Sheffer polynomials (2VGLSP) $_{{}_L}s_{n}(x,y)$. Thus, by using these substitutions in Theorems 2.1, 2.3 and 2.5, we obtain the corresponding results for the 2VLSP \cite{2SubNus}.\\

\noindent
{\bf Example 6.}  Since, for $z=0$, the LeGHP $\frac{{}_RH_n^{(r)}(x,y,z)}{n!}$ reduce to the 2-variable Legendre polynomials (2VLeP) $\frac{R_{n}(x,y)}{n!}$ (Table 1(VI)). Therefore, for the same choice of $z$, the  LeGHSP ${}_{{}_{r}LeG^{(r)}}s_{n}(x,y,z)$ reduce to the 2-variable Legendre based Sheffer polynomials (2VLeSP) $\frac{{}_{R}s_{n}(x,y)}{n!}$ \cite{2SubNusR}. Thus, by using these substitutions in Theorems 2.1, 2.3 and 2.5, we obtain the corresponding results  for the 2VLeSP $\frac{{}_{R}s_{n}(x,y)}{n!}$ \cite{2SubNusR}.\\

\noindent
{\bf Example 7.}  Since, for $x=0,y\rightarrow x, z\rightarrow yD_{y}y$, the LeGHP ${}_SH_{n}^{(r)}(0,x,yD_{y}y)$ reduce to the 2-variable truncated polynomials of order $r$ (2VTP) $=e_n^{(r)}(x,y)$ (Table 1(VII)). Therefore, for the same choice of $x$, $y$ and $z$, the  LeGHSP ${}_{{}_{s}LeG^{(r)}}s_{n}(x,y,z)$ reduce to the 2-variable truncated exponential based Sheffer polynomials (2VTESP) $_{e^{(r)}}s_{n}(x,y)$ \cite{2SubGaz}. Thus, by using these substitutions in Theorems 2.1, 2.3 and 2.5, we obtain the corresponding results  for the 2VTESP $_{e^{(r)}}s_{n}(x,y)$ \cite{2SubGaz}.\\

\noindent
{\bf Example 8.}  Since, for $r=2$, $x=0$, the LeGHP ${}_SH_{n}^{(r)}(0,y,z)$ reduce to the 2-variable Hermite Kamp\'{e} de F${\acute{e}}$riet polynomials (2VHKdFP) $H_{n}(y,z)$ (Table 1(VIII)). Therefore, for the same choice of $r$ and $x$, the  LeGHSP ${}_{{}_{s}LeG^{(r)}}s_{n}(x,y,z)$ reduce to the 2-variable Hermite Kamp${\rm\acute e}$ de F${\acute{e}}$riet based Sheffer polynomials (2VHKdFSP) ${}_{H}s_{n}(y,z)$ \cite{2SubSaad}. Thus, by using these substitutions in Theorems 2.1, 2.3 and 2.5, we obtain the corresponding results  for the 2VHKdFSP ${}_{H}s_{n}(y,z)$ \cite{2SubSaad}.\\

\noindent
{{\bf Example 9.}} 
 Since, for $r=2$; $x=0$ and for $r=2; x=0$,$y\rightarrow x, z\rightarrow
 y$ the LeGHP ${}_SH_n^{(r)}(x,y,z)$ and $\frac{{}_RH_n^{(r)}(x,y,z)}{n!}$ respectively reduce to the Hermite type polynomials (HTP) $G_{n}(x,y)$ (Table 1(IX)). Therefore, for the same choices, the  LeGHSP reduce to the Hermite type based Sheffer polynomials (HTSP) ${}_{G}s_{n}(x,y)$. Thus, by using these substitutions in Theorems 2.1, 2.3 and 2.5, we get the following results for the HTSP ${}_{G}s_{n}(x,y)$:\\

\noindent
\textbf{Table 3. Results for the HTSP ${}_{G}s_{n}(x,y)$}\\
\\
{\tiny{
		\begin{tabular}{lll}
			\hline
			\bf{S.} && \\
			\bf{No.} &~~~~~~~~~{\bf Results}&~~~~~~~~~~~~~~~~~~~~~~~{\bf Mathematical Expressions}\\
			
			\hline
			\bf{1.} & \bf{Generating } & $\frac {1}{g\big(f^{-1}(t)\big)}C_0\left(-xf^{-1}(t)\right)\exp\left(y\big(f^{-1}(t)\big)^2\right)$ \\
			\bf{}& \bf{function} &  $~~~~~~~~~~~~~~~~~~~~~~~~~~~~~~~~~~~~~~~~~=A(t)C_0\left(-xH(t)\right)\exp\left(y\big(H(t)\big)^2\right) =\sum_{n=0}^{\infty}{}_{G}s_{n}(x,y)\frac{t^n}{n!}$\\
			
			\hline
			\bf{2.} & \bf{Multiplicative and } & $\hat {M}=\left(D_x^{-1}+2y\frac{\partial}{\partial x}x\frac{\partial}{\partial x}-\frac {g^{\prime} \left(\partial_{x}x\partial_{x}\right)} {g\left(\partial_{x}x\partial_{x}\right)}\right)\frac {1} {f^{\prime}\left(\partial_{x}x\partial_{x}\right)}$\\
			\bf{}&&  $~~~~=\left(D_x^{-1}+2y\frac{\partial}{\partial x}x\frac{\partial}{\partial x}\right)H'\left(H^{-1}\left(\partial_{x}x\partial_{x}\right)\right)+\frac {A^{\prime} \left(H^{-1}\left(\partial_{x}x\partial_{x}\right)\right)}{A\left(H^{-1}\left(\partial_{x}x\partial_{x}\right)\right)}$,\\
			&\bf{derivative operators}&~~~~~~~~$\hat {P}=f\left(\partial_{x}x\partial_{x}\right)=H^{-1}\left(\partial_{x}x\partial_{x}\right)$\\
			
			\hline
			\bf{3.} & \bf{Differential } & $\left(\left(D_x^{-1}+2y\frac{\partial}{\partial x}x\frac{\partial}{\partial x}-\frac {g^{\prime} \left(\partial_{x}x\partial_{x}\right)} {g\left(\partial_{x}x\partial_{x}\right)}\right)\frac {f \left(\partial_{x}x\partial_{x}\right)} {f^{\prime}\left(\partial_{x}x\partial_{x}\right)}-n\right){}_{G}s_{n}(x,y)=0$, or equivalently\\
			\bf{}& \bf{equation} &  $\left(\left(\left(D_x^{-1}+2y\frac{\partial}{\partial x}x\frac{\partial}{\partial x}\right) H'\left(H^{-1}\left(\partial_{x}x\partial_{x}\right)\right) +\frac {A^{\prime} \left(H^{-1}\left(\partial_{x}x\partial_{x}\right)\right)}{A\left(H^{-1}\left(\partial_{x}x\partial_{x}\right)\right)}\right) H^{-1}\left(\partial_{x}x\partial_{x}\right)-n\right)$\\
			&&~~~~~~~~~~~~~~~~~~~~~~~~~~~~~~~~~~~~~~~~~~~~~~~~~~~~~~~~~~~~~~~~~~~~~~~~~~~~~~
			~~~~~~~~~~~~~~~~~~~~~$ {}_{G}s_{n}(x,y)=0$\\
			
			\hline
			\bf{4.} & \bf{Explicit} & ${}_{G}s_{n}(x,y)=\left(\left(D_x^{-1}+2y\frac{\partial}{\partial x}x\frac{\partial}{\partial x}-\frac {g^{\prime} \left(\partial_{x}x\partial_{x}\right)} {g\left(\partial_{x}x\partial_{x}\right)}\right)\frac {1} {f^{\prime}\left(\partial_{x}x\partial_{x}\right)}\right)^{n}\{1\}$\\
			\bf{}& \bf{representation} &  $~~~~~~~~~~~~=\left(\left(D_x^{-1}+2y\frac{\partial}{\partial x}x\frac{\partial}{\partial x}\right)H'\left(H^{-1}\left(\partial_{x}x\partial_{x}\right)\right)+\frac {A^{\prime} \left(H^{-1}\left(\partial_{x}x\partial_{x}\right)\right)}{A\left(H^{-1}\left(\partial_{x}x\partial_{x}\right)\right)}\right)^{n}\{1\}$\\
			
			\hline
			
\end{tabular}}}\\
\\

\noindent
{{\bf Example 10.}}  Since, for $x\rightarrow\frac{(x^{2}-1)}{4}$, $y\rightarrow x, z=0$ and for $r=1; x\rightarrow\frac{(1-x)}{2}$, $y\rightarrow \frac{(1+x)}{2},z=
0$, the LeGHP ${}_SH_{n}^{(r)}(\frac{(x^{2}-1)}{4},x,0)$ and ${}_RH_{n}^{(1)}(\frac{(1-x)}{2},\frac{(x+1)}{2},0)$ respectively reduce to the Legendre polynomials (LeP) $P_{n}(x)$ (Table 1(X)). Therefore for the same choices, the  LeGHSP respectively reduce to the Legendre based Sheffer polynomials (LeSP) ${}_{P}s_{n}(x)$. Thus, by Theorems 2.1, 2.3 and 2.5, we get the following results:

\noindent
\textbf{Table 4. Results for the LeSP ${}_{P}s_{n}(x)$}
\\
{\tiny{
		\begin{tabular}{lll}
			\hline
			\bf{S.} && \\
			\bf{No.} &~~~~~~~~~{\bf Results}&~~~~~~~~~~~~~~~~~~~~~~~{\bf Mathematical Expressions}\\
			
			\hline
			\bf{1.} & \bf{Generating } & $\frac {1}{g\big(f^{-1}(t)\big)}C_0\left(-\left(\frac {x^{2}-1}{4}\right)\big(f^{-1}(t)\big)^2\right)\exp\left(xf^{-1}(t)\right)$ \\
			\bf{}& \bf{function} &  $~~~~~~~~~~~~~~~~~~~~~~~~~~~~~~~~=A(t)C_0\left(-\left(\frac {x^{2}-1}{4}\right)\big(H(t)\big)^2\right)\exp\left(xH(t)\right) =\sum_{n=0}^{\infty}{}_{P}s_{n}(x)\frac{t^n}{n!}$\\
			
			\hline
			\bf{2.} & \bf{Multiplicative and } & $\hat {M}=\left(x+2D_{\left(\frac {x^{2}-1}{4}\right)}^{-1}\frac{\partial}{\partial x}-\frac {g^{\prime} \left(\partial_{x}\right)} {g\left(\partial_{x}\right)}\right)\frac {1} {f^{\prime}\left(\partial_{x}\right)}$\\
			\bf{}&&  $~~~~=\left(x+2D_{\left(\frac {x^{2}-1}{4}\right)}^{-1}\frac{\partial}{\partial x}\right) H'\left(H^{-1}\left(\partial_{x}\right)\right)+\frac {A^{\prime} \left(H^{-1}\left(\partial_{x}\right)\right)}{A\left(H^{-1}\left(\partial_{x}\right)\right)}$,\\
			&\bf{derivative operators}&~~~~~~~~$\hat {P}=f\left(\partial_{x}\right)=H^{-1}\left(\partial_{x}\right)$\\

			\hline
			\bf{3.} & \bf{Differential } & $\left(\left(x+2D_{\left(\frac {x^{2}-1}{4}\right)}^{-1}\frac{\partial}{\partial x}-\frac {g^{\prime} \left(\partial_{x}\right)} {g\left(\partial_{x}\right)}\right)\frac {f \left(\partial_{x}\right)} {f^{\prime}\left(\partial_{x}\right)}-n\right){}_{P}s_{n}(x)=0$, or equivalently\\
			\bf{}& \bf{equation} &  $\left(\left(\left(x+2D_{\left(\frac {x^{2}-1}{4}\right)}^{-1}\frac{\partial}{\partial x}\right) H'\left(H^{-1}\left(\partial_{x}\right)\right) +\frac {A^{\prime} \left(H^{-1}\left(\partial_{x}\right)\right)}{A\left(H^{-1}\left(\partial_{x}\right)\right)}\right) H^{-1}\left(\partial_{x}\right)-n\right){}_{P}s_{n}(x)=0$~~~~~~~~~~~~~\\
			
			\hline
			\bf{4.} & \bf{Explicit} & ${}_{P}s_{n}(x)=\left(\left(x+2D_{\left(\frac {x^{2}-1}{4}\right)}^{-1}\frac{\partial}{\partial x}-\frac {g^{\prime} \left(\partial_{x}\right)} {g\left(\partial_{x}\right)}\right)\frac {1} {f^{\prime}\left(\partial_{x}\right)}\right)^{n}\{1\}$\\
			\bf{}& \bf{representation} &  $~~~~~~~~~~~~=\left(\left(x+2D_{\left(\frac {x^{2}-1}{4}\right)}^{-1}\frac{\partial}{\partial x}\right)H'\left(H^{-1}\left(\partial_{x}\right)\right)+\frac {A^{\prime} \left(H^{-1}\left(\partial_{x}\right)\right)}{A\left(H^{-1}\left(\partial_{x}\right)\right)}\right)^{n}\{1\}$\\
			
			\hline
			
\end{tabular}}}\\
\\

\noindent
{{\bf Example 11.}} Since, for $r=3$, $x\rightarrow z\partial_{z}z$, $y\rightarrow x$, $z\rightarrow y$, the LeGHP ${}_{S}H_{n}^{(r)}(x,y,z)$ reduce to the Bell type polynomials (BTP) $H_{n}^{(3,2)}(x,y,z)$ (Table 1(XI)). Therefore, for the same substitutions, the  LeGHSP ${}_{{}_{s}LeG^{(r)}}s_{n}(x,y,z)$ reduce to the Bell type based Sheffer polynomials (BTSP) $_{H^{(3,2)}}s_{n}(x,y,z)$. Thus, by Theorems 2.1, 2.3 and 2.5, we get the following results:\\

\noindent
\textbf{Table 5. Results for the BTSP $_{H^{(3,2)}}s_{n}(x,y,z)$}\\
\\
{\tiny{
		\begin{tabular}{lll}
			\hline
			\bf{S.} && \\
			\bf{No.} &~~~~~~~~~{\bf Results}&~~~~~~~~~~~~~~~~~~~~~~~{\bf Mathematical Expressions}\\
			
			\hline
			\bf{1.} & \bf{Generating } & $\frac {1}{g\big(f^{-1}(t)\big)}\exp\left(xf^{-1}(t)+y\big(f^{-1}(t)\big)^3+z\big(f^{-1}(t)\big)^2\right)$ \\
			\bf{}& \bf{function} &  $~~~~~~~~~~~~~~~~~~~=A(t)\exp\left(xH(t)+y\big(H(t)\big)^3+z\big(H(t)\big)^2\right)
			=\sum_{n=0}^{\infty}{_{H^{(3,2)}}s_{n}(x,y,z)}\frac{t^n}{n!}$\\
			
			\hline
			\bf{2.} & \bf{Multiplicative and } & $\hat {M}=\left(x+3y\frac{\partial^{2}}{\partial x^{2}}+2z\frac{\partial}{\partial x}-\frac {g^{\prime} \left(\partial_{x}\right)} {g\left(\partial_{x}\right)}\right)\frac {1} {f^{\prime}\left(\partial_{x}\right)}$\\
			\bf{}&&  $~~~~=\left(x+3y\frac{\partial^{2}}{\partial x^{2}}+2z\frac{\partial}{\partial x}\right) H'\left(H^{-1}\left(\partial_{x}\right)\right)+\frac {A^{\prime} \left(H^{-1}\left(\partial_{x}\right)\right)}{A\left(H^{-1}\left(\partial_{x}\right)\right)}$,\\
			&\bf{derivative operators}&~~~~~~~~$\hat {P}=f\left(\partial_{x}\right)=H^{-1}\left(\partial_{x}\right)$\\
			
			\hline
			\bf{3.} & \bf{Differential } & $\left(\left(x+3y\frac{\partial^{2}}{\partial x^{2}}+2z\frac{\partial}{\partial x}-\frac {g^{\prime} \left(\partial_{x}\right)} {g\left(\partial_{x}\right)}\right)\frac {f \left(\partial_{x}\right)} {f^{\prime}\left(\partial_{x}\right)}-n\right){_{H^{(3,2)}}s_{n}(x,y,z)}=0$, or equivalently\\
			\bf{}& \bf{equation} &  $\left(\left(\left(x+3y\frac{\partial^{2}}{\partial x^{2}}+2z\frac{\partial}{\partial x}\right) H'\left(H^{-1}\left(\partial_{x}\right)\right) +\frac {A^{\prime} \left(H^{-1}\left(\partial_{x}\right)\right)}{A\left(H^{-1}\left(\partial_{x}\right)\right)}\right) H^{-1}\left(\partial_{x}\right)-n\right){_{H^{(3,2)}}s_{n}(x,y,z)}=0$\\
			
			\hline
			\bf{4.} & \bf{Explicit} & $_{H^{(3,2)}}s_{n}(x,y,z)=\left(\left(x+3y\frac{\partial^{2}}{\partial x^{2}}+2z\frac{\partial}{\partial x}-\frac {g^{\prime} \left(\partial_{x}\right)} {g\left(\partial_{x}\right)}\right)\frac {1} {f^{\prime}\left(\partial_{x}\right)}\right)^{n}\{1\}$\\
			\bf{}& \bf{representation} &  $~~~~~~~~~~~~~~~~~~~~~~~~~~~~~~~~~~=\left(\left(x+3y\frac{\partial^{2}}{\partial x^{2}}+2z\frac{\partial}{\partial x}\right) H'\left(H^{-1}\left(\partial_{x}\right)\right)+\frac {A^{\prime} \left(H^{-1}\left(\partial_{x}\right)\right)}{A\left(H^{-1}\left(\partial_{x}\right)\right)}\right)^{n}\{1\}$\\
			
			\hline
			
\end{tabular}}}\\
\\

\noindent
{\bf{Remark~3.1.}}~For $\frac{1}{g(f^{-1}(t))}=A(t)=1$, the above mentioned special cases (Examples 1-11) of  the  LeGHSP ${}_{{}_{s}LeG^{(r)}}s_{n}(x,y,z)$ and $\frac{{}_{{}_{r}LeG^{(r)}}s_{n}(x,y,z)}{n!}$  yield the corresponding results for the LGHASP ${}_{{}_{s}LeG^{(r)}}\mathfrak{s}_{n}(x,y,z)$ and $\frac{{{}_{r}LeG^{(r)}}s_{n}(x,y,z)}{n!}$.\\

\noindent
{\bf{Remark~3.2.}}~For $f^{-1}(t)=H(t)=t$, the above mentioned special cases of  the  LeGHSP yield the corresponding results for the Legendre-Gould Hopper based Appell polynomials LeGHAP ${}_{{}_{s}LeG^{(r)}}\mathcal{A}_{n}(x,y,z)$ and $\frac{{}_{{}_{r}LeG^{(r)}}\mathcal{A}_{n}(x,y,z)}{n!}$.\\

In the next section, we derive certain operational and integral representations for the  LGHSP ${}_{{}_{s}LeG^{(r)}}s_{n}(x,y,z)$ and $\frac{{}_{{}_{r}LeG^{(r)}}s_{n}(x,y,z)}{n!}$.\\

\vspace{.50cm}
\section{Operational and integral representations}

To establish the operational representations for the  LeGHSP ${}_{{}_{s}LeG^{(r)}}s_{n}(x,y,z)$ and $\frac{{}_{{}_{r}LeG^{(r)}}s_{n}(x,y,z)}{n!}$, we prove the following results:\\

\noindent
\begin{theorem}
The following operational representation connecting the LeGHSP ${}_{{}_{s}LeG^{(r)}}s_{n}(x,y,z)$ and $\frac{{}_{{}_{r}LeG^{(r)}}s_{n}(x,y,z)}{n!}$ with the Sheffer polynomials $s_{n}(x)$ holds true:
\begin{equation}
{}_{{}_{s}LeG^{(r)}}s_{n}(x,y,z)=\exp\left(D_{x}^{-1}\frac{\partial^{2}}{\partial y^{2}}+z\frac{\partial^{r}}{\partial y^{r}}\right)s_{n}(y)
\end{equation}
and
\begin{equation}
{}_{{}_{r}LeG^{(r)}}s_{n}(x,y,z)=\exp\left(-D_{x}^{-1}\frac{\partial}{\partial y}+D_{y}^{-1}\frac{\partial}{\partial y}+z\frac{\partial^{r}}{\partial y^{r}}\right)s_{n}(0).
\end{equation}
\end{theorem}
\noindent
{\bf{Proof.}}~
In view of equation \eqref{2.8}, the proof is direct use of identity \eqref{crofton}.\\

\begin{theorem}
The following operational representation connecting the LeGHSP ${}_{{}_{s}LeG^{(r)}}s_{n}(x,y,z)$  with the 2VLebSP $_{{}_2L}s_{n}(x,y)$ holds true:
\begin{equation}\label{4.3}
{}_{{}_{s}LeG^{(r)}}s_{n}(x,y,z)=\exp\left(z\frac{\partial^{r}}{\partial y^{r}}\right){_{{}_2Leb}s_{n}(x,y)}
\end{equation}
\end{theorem}
\begin{proof}
From equation \eqref{2.1a}, we have
\begin{equation}\label{4.4}
\frac{\partial^{r}}{\partial y^{r}}{{}_{{}_{s}LeG^{(r)}}s_{n}(x,y,z)}=\frac{\partial}{\partial z}{{}_{{}_{s}LeG^{(r)}}s_{n}(x,y,z)}.
\end{equation}

Since, in view of Table 1(II), we have
\begin{equation}
_SH^{(r)}_{n}(x,y,0)={}_2{L}_{n}(x,y).
\end{equation}
Therefore, from Example 2 of Section 3, we have
\begin{equation}\label{4.6}
{}_{{}_{s}LeG^{(r)}}s_{n}(x,y,0)={_{{}_2Leb}s_{n}(x,y)}
\end{equation}

Now, solving equations \eqref{4.4} subject to initial condition \eqref{4.6}, we get assertion \eqref{4.3}.
\end{proof}
\noindent

\begin{theorem}
The following operational representation connecting the LGHSP $_{{}_LH^{(m,r)}}s_{n}(x,y,z)$  with the GHSP ${}_{H^{(r)}}s_{n}(y,z)$ holds true:
\begin{equation}
_{{}_LH^{(m,r)}}s_{n}(x,y,z)=\exp\left(D_{x}^{-1}\frac{\partial^{m}}{\partial y^{m}}\right){}_{H^{(r)}}s_{n}(y,z).
\end{equation}
\end{theorem}
\noindent
{\bf{Proof.}}~From equations (1.15) and (2.1) (or (2.2)), we have\\
$$\frac{\partial^{m}}{\partial y^{m}}~{_{{}_LH^{(m.r)}}s_{n}(x,y,z)}=\frac{\partial}{\partial D_{x}^{-1}}~{_{{}_LH^{(m,r)}}s_{n}(x,y,z)},\eqno(4.7)$$
where (\cite{2GHMDat}; p. 32 (8)):\\
$$\frac{\partial}{\partial D_{x}^{-1}}:=\frac{\partial}{\partial x}x\frac{\partial}{\partial x}.$$

Since, in view of Table 1(IV), we have
$${}_LH_{n}^{(m,r)}(0,y,z)=H^{(r)}(y,z).\eqno(4.8)$$

Therefore, from Example 4 of Section 3, we have
$$_{{}_LH^{(m,r)}}s_{n}(0,y,z)=_{H^{(r)}}s_{n}(y,z).\eqno(4.9)$$

Solving equation (4.7) subject to initial condition (4.9), we get assertion (4.6).\\


\newpage
\noindent

Next, we prove the integral representations for the  LeGHSP ${}_{{}_{s}LeG^{(r)}}s_{n}(x,y,z)$ and $\frac{{}_{{}_{r}LeG^{(r)}}s_{n}(x,y,z)}{n!}$ in the form of following theorems:\\

\begin{theorem}\label{t4.4} 
The following integral representation for the  LeGHSP ${}_{{}_{s}LeG^{(r)}}s_{n}(x,y,z)$ holds true:
\begin{equation}\label{4.8}
{}_{{}_{s}LeG^{(r)}}s_{n}(x,y,z)=\int_{0}^{\infty} e^{-s} {{}_{{}_{s}LeG^{(r)}}s_{n}(x,y,sD_{x}^{-1})}\,ds.
\end{equation}
\end{theorem}
\begin{proof}
In view of \eqref{leGHP1} and \eqref{2.1a}, we write

\begin{equation}\label{4.11}
\sum_{n=0}^{\infty}{{}_{{}_{s}LeG^{(r)}}s_{n}(x,y,z)}\frac{t^n}{n!}=\frac {1}{g\big(f^{-1}(t)\big)}\sum_{n=0}^{\infty}{}_{S}H_{n}^{(r)}(x,y,z)\frac{\big(f^{-1}(t)\big)^n}{n!}.
\end{equation}

Using the following integral representation of LeGHP ${}_{S}H_{n}^{(r)}(x,y,z)$ \cite{ghazala}:
\begin{equation}
{}_SH_{n}^{(r)}(x,y,z)~=~\int_{0}^{\infty} e^{-s} ~{}_SH_{n}^{(r)}(x,y,sD_{z}^{-1}) ~ds~,
\end{equation}
in the r.h.s. of equation \eqref{4.11}, we find
\begin{equation}
\sum_{n=0}^{\infty}{{}_{{}_{s}LeG^{(r)}}s_{n}(x,y,z)}\frac{t^n}{n!}=\frac {1}{g\big(f^{-1}(t)\big)}\int_{0}^{\infty} e^{-s} \left(\sum_{n=0}^{\infty}{}_{S}H_{n}^{(r)}(x,y,sD_{z}^{-1})\frac{\big(f^{-1}(t)\big)^n}{n!}\right) ~ds.
\end{equation}

Now, making use of the generating function \eqref{leGHP1} in the r.h.s. of the above equation, we have
\begin{equation*}
\sum_{n=0}^{\infty}{{}_{{}_{s}LeG^{(r)}}s_{n}(x,y,z)}\frac{t^n}{n!}=\int_{0}^{\infty} e^{-s} \left(\frac {1}{g\big(f^{-1}(t)\big)}\,C_0(-x(f^{-1}(t))^{2})\exp(y(f^{-1}(t))+sD_{z}^{-1}(f^{-1}(t))^{r}\right) ~ds,
\end{equation*}
which in view of\eqref{2.1a} yields
\begin{equation}
\sum_{n=0}^{\infty}{{}_{{}_{s}LeG^{(r)}}s_{n}(x,y,z)}\frac{t^n}{n!}=\sum_{n=0}^{\infty}\,\left(\int_{0}^{\infty} e^{-s}{{}_{{}_{s}LeG^{(r)}}s_{n}(x,y,sD_{z}^{-1})}\,ds \right)\frac{t^{n}}{n!}
\end{equation}

Finally, equating the coefficients of like powers of $t$ above, we get assertion \eqref{4.8}.
\end{proof}
\begin{theorem}\label{t4.5}
The following integral representations for $\frac{{}_{{}_{r}LeG^{(r)}}s_{n}(x,y,z)}{n!}$ hold true:
\begin{equation}
{}_{{}_{r}LeG^{(r)}}s_{n}(x,y,z)=\int_{0}^{\infty} e^{-s} {{}_{{}_{r}LeG^{(r)}}s_{n}(x,y,sD_{x}^{-1})}\,ds.
\end{equation}
\end{theorem}
\begin{proof}
The proof of Theorem \ref{t4.4} is same as of Theorem \ref{t4.5}. So we omit all the details.
\end{proof}
\vspace{.35cm}
In Section 3, we have obtained the results for the new and known families of special polynomials related to Sheffer sequences by taking the special cases of the LGHP $_LH^{(m,r)}_{n}(x,y,z)$. In the Appendix section, we consider certain special polynomials belonging to the Sheffer and associated Sheffer families and obtain the results for the corresponding mixed special polynomials.\\

\noindent
{\bf{5.~~Appendix}}\\

The Sheffer class contains important sequences such as the Hermite, Laguerre, Bernoulli, Poisson-Charlier polynomials {\it etc}. These polynomials are important from the view point of applications in physics and number theory. Also, the associated Sheffer family contains Mittag-Leffler, exponential, lower factorial polynomials {\it etc}.\\

We present the lists of some known members of the Sheffer and associated Sheffer families in Tables 12 and 13 respectively.\\

\noindent
\textbf{Table 12. Some members of the Sheffer family}\\
\\
{\tiny{
		\begin{tabular}{lllll}
			\hline
			&&&&\\
			\bf{S. No.} & \bf{$g(t)$; $A(t)$} & \bf{$f(t)$; $H(t)$} & \bf{Generating Functions} & \bf{Polynomials}\\
			\hline
			I. & $e^{(\frac{t}{\nu})^{k}}$; ~$e^{-t^{k}}$ &$\frac{t}{\nu}$; ~$\nu t$ & $\exp(\nu xt-t^k)$&Generalized Hermite \\
			&&&$=\sum\limits_{n=0}^{\infty} H_{n,k,\nu}(x)\frac{t^n}{n!}$&polynomials \\
			&&&&$H_{n,k,\nu}(x)$ \cite{2Lah}\\
			\hline
			II. & $(1-t)^{-\alpha-1}$; ~$(1-t)^{-\alpha-1}$ & $\frac{t}{t-1}$; ~$\frac{t}{t-1}$ &$\frac{1}{(1-t)^{\alpha+1}}\exp (\frac{xt}{t-1})$ &Generalized Laguerre \\
			&&&$=\sum\limits_{n=0}^{\infty} L_{n}^{(\alpha)}(x)t^{n}$&polynomials\\
			&&&&$n!L_{n}^{\alpha}(x)$ \cite{2And, 2Rain}\\
			\hline
			III.& $\frac{2}{e^t-1}$; ~$\frac{t}{1-t}$ & $\frac{e^t-1}{e^t+1}$; ~$\ln (\frac{1+t}{1-t})$ & $\frac{t}{1-t}(\frac{1+t}{1-t})^{x}$& Pidduck polynomials\\
			&&&$=\sum\limits_{n=0}^{\infty} P_{n}(x)\frac{t^n}{n!}$ &$P_{n}(x)$ \cite{2Boas, 2Erd}\\
			\hline
			IV.& $(1-t)^{-\beta}$; ~$e^{\beta t}$ & $\ln (1-t)$; ~$1-e^t$ & $\exp (\beta t+x(1-e^t))$& Acturial polynomials \\
			&&&$=\sum\limits_{n=0}^{\infty} a_{n}^{(\beta)}(x)\frac{t^n}{n!}$&$a_{n}^{(\beta)}(x)$ \cite{2Boas}\\
			\hline
			V.& $\exp (a(e^t-1))$; ~$e^{-t}$ & $a(e^t-1)$; ~$\ln (1+\frac{t}{a})$ & $e^{-t}(1+\frac{t}{a})^{x}$& Poisson-Charlier \\
			&&&$=\sum\limits_{n=0}^{\infty} c_{n}(x;a)\frac{t^n}{n!}$&polynomials\\
			&&&&$c_{n}(x;a)$ \cite{2Erde, 2Jor, 2Sze}\\
			\hline
			VI. & $(1+e^{\lambda t})^{\mu}$; ~$(1+(1+t)^{\lambda})^{-\mu}$ & $e^t-1$; ~$\ln (1+t)$ & $(1+(1+t)^{\lambda})^{-\mu}(1+t)^{x}$& Peters polynomials\\
			&&&$=\sum\limits_{n=0}^{\infty} S_{n}(x;\lambda,\mu)\frac{t^n}{n!}$&$S_{n}(x;\lambda,\mu)$ \cite{2Boas}\\
			\hline
			VII.& $\frac{t}{e^t-1}$; ~$\frac{t}{\ln (1+t)}$ & $e^t-1$; ~$\ln (1+t)$ & $\frac{t}{\ln (1+t)}(1+t)^{x}$ & Bernoulli polynomials \\
			&&&$=\sum\limits_{n=0}^{\infty} b_{n}(x)\frac{t^n}{n!}$&of the second kind \\
			&&&&$b_{n}(x)$ \cite{2Jor}\\
			\hline
			VIII.& $\frac{1}{2}(1+e^t)$; ~$\frac{2}{2+t}$ &$e^t-1$; ~$\ln (1+t)$ & $\frac{2}{2+t}(1+t)^{x}$& Related polynomials\\
			&&&$=\sum\limits_{n=0}^{\infty} r_{n}(x)\frac{t^n}{n!}$&$r_{n}(x)$ \cite{2Jor}\\
			\hline
			IX.& $\sec t$; ~$\frac{1}{\sqrt{1+t^{2}}}$ & $\tan t$; ~$\arctan(t)$ & $\frac{1}{\sqrt{1+t^{2}}}\exp (x \arctan(t))$& Hahn polynomials\\
			&&&$=\sum\limits_{n=0}^{\infty} R_{n}(x)\frac{t^n}{n!}$&$R_{n}(x)$ \cite{2Bend}\\
			\hline
			X.& $\frac{1+t}{(1-t)^{a}}$; ~$(1-4t)^{-\frac{1}{2}}\left(\frac{2}{1+\sqrt{1-4t}}\right)^{a-1}$ & $\frac{1}{4}-\frac{1}{4}\left(\frac{1+t}{1-t}\right)^{2}$; ~$\frac{-4t}{(1+\sqrt{1-4t})^{2}}$ & $(1-4t)^{-\frac{1}{2}}\left(\frac{2}{1+\sqrt{1-4t}}\right)^{a-1} $& Shively's psedo-Laguerre\\
			&&&$\times\exp \left(\frac{-4xt}{(1+\sqrt{1-4t})^{2}}\right)$&polynomials\\
			&&&$=\sum\limits_{n=0}^{\infty} R_{n}(a,x)t^n$ & $R_{n}(a,x)$ \cite{2Rain}\\
			\hline
\end{tabular}}}\\
\\

\noindent \textbf{Table 13.  Some members of the associated Sheffer family}\\
\\
{\tiny{
		\begin{tabular}{llll}
			\hline
			&&&\\
			\bf{S. No.} & ~~~~~~~\bf{$f(t)$; $H(t)$} &~~~~~~~~ \bf{Generating Functions}& ~~~~~~~~\bf{Polynomials}~~~~~~~~\\

			\hline
			I.& $\frac{e^t-1}{e^t+1}$; $\ln\left(\frac{1+t}{1-t}\right)$ & $\left(\frac{1+t}{1-t}\right)^x=\sum_{n=0}^{\infty}M_{n}(x)\frac{t^n}{n!}$ & Mittag-Leffler polynomials $M_{n}(x)$ \cite{2Bat} \\
			
			\hline
			&&&\\
			II.& $\ln(1+t)$; $e^t-1$ & $\exp(x(e^t-1))=\sum_{n=0}^{\infty}\phi_{n}(x)\frac{t^n}{n!}$ & Exponential polynomials $\phi_{n}(x)$ \cite{2Bell}  \\
			
			\hline
			&&&\\
			III.&  $e^t-1$ ; $\ln(1+t)$& $(1+t)^{x}=\sum_{n=0}^{\infty}(x)_{n}\frac{t^n}{n!}$ & Lower factorial polynomials $(x)_{n}$ \cite{2SRom} \\
			
			\hline
			&&&\\
			IV.&  $-\frac{1}{2}t^{2}+t$ ; $1-\sqrt {1-2t}$& $\exp \left(x(1-\sqrt {1-2t})\right)=\sum_{n=0}^{\infty}p_{n}(x)\frac{t^n}{n!}$ & Bessel polynomials $p_{n}(x)$ \cite{2Car, 2Krall} \\
			
			\hline
			
\end{tabular}}}\\
\\

\noindent
{\bf{Remark~5.1.}}~We remark that corresponding to each member belonging to the Sheffer (or associated Sheffer) family, there exists a new special polynomial belonging to the LeGHSP (or LeGHASP) family. The generating function and other properties of these special polynomials can be obtained from the results derived in Section 2.\\

Thus, by taking $g(t)$ (or $A(t)$) and $f(t)$ (or $H(t)$) of the special polynomials belonging to Sheffer family (Table 12 (I to X)) in equations \eqref{2.1a} (or \eqref{2.1b}), \eqref{2.9} (or \eqref{2.10}), \eqref{2.14} (or \eqref{2.15}), \eqref{2.16} (or \eqref{2.17}), \eqref{2.30} and \eqref{2.31}, we can get the generating function, multiplicative and derivative operators for the corresponding members belonging to the LeGHSP family.\\

\vspace{2.5cm}
\noindent

\noindent

\end{document}